\documentclass[11pt,twoside]{article}
\usepackage{amsmath,     %
            amssymb,     %
            amsthm,      %
            enumerate,   %
            epsfig,      %
            graphics,    %
            indentfirst, %
            setspace}    %

\usepackage{color}
\newtheorem{theorem}{Theorem}[section]          %
\newtheorem{lemma}{Lemma}[section]              %
\def\OL#1{\overline {#1}}

\def\UL#1{\underline{#1}}

\let\wt\widetilde

\let\bo\partial

\def\7{\mathaccent"7017}

\def\q{\quad}

\def\Q{\qquad}

\newcommand{\bi}[1]{\mbox{\boldmath $ {#1} $}} 

\newcommand{\mbb} {\mathbb}

\newcommand{\mc}[1]{\mbox{$\mathcal {#1}$}}

\newcommand{\mr} {\mathrm}

\newcommand {\mbf} {\mathbf}

\begin{document}

\title{ {\bf  Robust consistent a posteriori  error majorants  for approximate solutions of diffusion-reaction equations} }
 \vskip -5mm
\author{V. G. Korneev
}
 \maketitle
 \vspace*{-1cm}

\begin{center}
\small{   St. Petersburg State  University,
         Russia } \\
  \small{\tt Vad.Korneev2011@yandex.ru
  }
 \end{center}


\begin{abstract}
   Efficiency of the error control of numerical solutions of  partial differential equations entirely depends on the two factors: accuracy of an a posteriori error majorant and the computational cost of its
  evaluation for some test function/vector-function plus the cost of the latter. In the paper, consistency of an a posteriori bound implies that it is the same in the order with the respective unimprovable a priori  bound. Therefore, it is the basic characteristic related to the first factor. The paper is dedicated to the elliptic diffusion-reaction equations. We present a guaranteed robust a posteriori error majorant effective at any nonnegative constant reaction coefficient  (r.c.). For a wide range of finite element solutions on a quasiuniform meshes the majorant is consistent. For  big values of  r.c. the majorant coincides with the majorant of Aubin (1972),  which, as it is known, for  not big r.c. ($<ch^{-2}$) is inconsistent and loses its sense  at r.c. approaching zero.  Our  majorant improves also some  other majorants derived for the Poisson and reaction-diffusion equations.
 \end{abstract}

 \section{\normalsize Introduction}\label{Se:In}
 \setcounter{equation}{0}

\thispagestyle{empty}
 \par For the successful  error control of approximate solutions to the boundary value problems, the guaranteed a posteriori error majorant must be sufficiently accurate and cheap in a sense of the computational work. The first requirement can be considered as satisfied at least in part, if the majorant is consistent in respect to  the order of accuracy with  the  a priori convergence estimate of the numerical method. Obviously, a consistent majorant is unimprovable in the order, if the  a priori convergence estimate is unimprovable in the same sense.  An error majorant usually depends on the approximate solution and on some other function or functions which are termed test functions.  In this paper consistency assumes that it can be approved by an easily calculated test function  with the use of some procedure of the linear complexity.
  \par The term "functional a posteriory error majorants"\, is usually related to a posteriory error bounds possessing significant generality and some other positive properties. However, sometimes generality is attained for the price of the lost of consistency, resulting in the necessity of attracting some majorant minimization procedures over the space of admissible  test functions \cite{CarstensenMerdon:2013, FrolovNeittaanmakiRepin:2003}.  The computational cost of such procedures can exceed
 the cost of the  numerical solution of the boundary value problem.
   \par The majorant of Aubin \cite{Aubin:72}  is one of the  earliest. Let us illustrate it on a model problem
  \begin{equation}\label{poiss}
  \begin{array}{c}
 - {\mr{div }} ({\bf A} {\mr{grad}} \, u) + \sigma u=f(x),\q x\in\Omega \subset {\mbb R}^m\,, \\
   u\big|_{\Gamma_D}=\psi_D\,, \q
   - {\bf A} \nabla \,u\cdot {\bi \nu} \big|_{\Gamma_N}=\psi_N\,,
  \end{array}
  \end{equation}
  where $\Gamma_D,\,\,\Gamma_N$ are not intersecting parts of the boundary
  $ \bo \Omega= \Gamma_D\cup \Gamma_N,\,\,{\mr{mes}}\, \Gamma_D>0$, ${\bi \nu}$ is the internal unite normal to the boundary, ${\bf A} $ is a symmetric
  $m\times m$ matrix, and $0<\sigma={\mr{const}} $.  It is assumed that the matrix ${\bf A} $
    stisfies the inequalities
 $$
\mu_1   {\bi \xi}\cdot {\bi \xi}\le {\bf A} {\bi \xi}\cdot {\bi \xi}\le \mu_2\, {\bi \xi}\cdot {\bi \xi} \,, \q 0<\mu_1,\mu_2={\mr{const}},
 $$
   for any $x\in\Omega$ and  ${\bi \xi}\in {\mbb R}^m$.  Everywhere in the paper, the boundary $\bo\Omega$,
    the coefficients of the matrix ${\bf A} $, and the right hand part $f$ are assumed to be sufficiently smooth, if more specific requirements to their smoothness are absent.

    \par  The error bounds in the energy norm
 \begin{equation}\label{E-norm}
    |\!\!\,\! |\!\!\,\!| u|\!\,\!\!|\!\,\!\!|=(\|u\|_{\bf A}^2+\sigma \|u\|_{L_2(\Omega)}^2)^{1/2}\,,
    \Q \|u\|_{\bf A}^2=\int_\Omega \nabla  u\cdot {\bf A}\nabla u\,,
   \end{equation}
     are most usable in applications.  For  vectors  ${\bf y}\in {\mbb R}^m$  we introduce also the spaces ${\mbf L}_2(\Omega)=(L_2(\Omega))^m$, $ {\bf H}(\Omega, {\mr{div}})=\{ {\bf y}\in {\bf L}_2(\Omega): {\mr{div}}\, {\bf y} \in L_2(\Omega)\}$ and the norm
      $]\!| {\bf y} |\![_{{\bf A}^{-1}}=(\int_\Omega  {\bf A}^{-1} {\bf y}\cdot {\bf y})^{1/2}$.
 \begin{theorem}\label{Th:Aubin-1}
  Let  $f\in L_2(\Omega),\,\psi_D\in H^1(\Omega),\,\,\psi_N\in L_2(\Gamma_N)$,  $v$
  is any function from
   $H^1(\Omega)$  satisfying the boundary condition on $\Gamma_D$. Then for any
  ${\bf z} \in {\bf H}(\Omega, {\mr{div}})$, satisfying on  ${\Gamma_N}$
   the boundary condition  ${\bf z}\cdot {\bi \nu}
  =\psi_N$, we have
 \begin{equation}\label{Aubin-1}
               |\!\!\,\!|\!\!\,\!| v-u|\!\,\!\!|\!\,\!\!|^2\le ]\!| {\bf A}\nabla\,v +{\bf z} |\![_{{\bf A}^{-1}}^2+
               \frac{1}{\sigma}\|f-\sigma v - {\mr{div}} \,{\bf z}\|_{L_2(\Omega)}^2\,.
 \end{equation}
  \end{theorem}
  \par  The bound  (\ref{Aubin-1}) is   a particular case  of the results of \cite{Aubin:72}, see, {\em e. g.}, Theorem  22 in Introduction and additionally
   Theorems 1.2, 1.4, 1.6 of ch. 10.
     It can be found also in  \cite{Repin:2000, RepinSauter:2006}.  Obviously, this majorant becomes meaningless at  $ \sigma\rightarrow 0$.  Let us add that in (\ref{Aubin-1}) one can use ${\bf z}=- {\bf A} \nabla \,w$ with any $w\in H^1(\Omega, {\mc L})$, satisfying the boundary condition in (\ref{poiss}) on ${\Gamma_N}$.  Here ${\mc L}=- {\mr{div }} ({\bf A} {\mr{grad}} )$ and $ H^1(\Omega, {\mc L})=\{w:w\in H^1(\Omega),\,\,{\mc L}w\in L_2(\Omega)\}  $. If
     $v\in C(\Omega)\cap H^1(\Omega)$ is the finite element  solution, then $w=\wt{w}(v)$ most often is obtained from $v$ by some recovery technique [13, 1, 4, 5].

   \par Let $ \7H^1(\Omega):=\{v\in H^1(\Omega):v|_{\bo \Omega}=0 \}$, for simplicity $\Gamma_D=\bo \Omega$,
      $\psi_D\equiv 0$, ${\bf A}={\bf I}$, where  ${\bf I}$ is the unity matrix. In \cite{RepinFrolov:02} for the case $ \sigma= 0$, it was suggested the majorant
   \begin{equation}\label{rep-frol}
   \|\nabla (v-  u) \|_{{\bf L}_2(\Omega)}^2 \le (1+\epsilon)\|\nabla v+ {\bf
    z} \|_{{\bf L}_2(\Omega)}^2  + c_\Omega(1+\frac{1}{\epsilon})\|\nabla\cdot {\bf z}- {
    f}\|_{ L_2(\Omega)}^2\,,\q \forall \,\,\epsilon>0\,,
\end{equation}
    where $v$ and ${\bf z}$ are any function and vector-function from  $\7H^1(\Omega)$ and ${\bf H}(\Omega, {\mr{div}})$, respectively,
     and $c_\Omega$ is the constant from the Friedrichs inequality
     \par  Attempts to modify the Aubin's majorant in such a way  that it provided admissible accuracy for all $ \sigma\ge 0$
     were made in the papers \cite{RepinSauter:2006, Churilova:14}.  The latter suggests the majorant for all  $\sigma={\mr{const}}\ge 0$ of the form
    \begin{equation}\label{Chur-1}
               |\!\!\,\!|\!\!\,\!| v-u|\!\,\!\!|\!\,\!\!|^2\le (1+\epsilon)]\!| {\bf A}\nabla\,v +{\bf z} |\![_{{\bf A}^{-1}}^2+
               \frac{1}{\sigma+\frac{\epsilon}{c_\Omega(1+\epsilon)}}\|f-\sigma v - {\mr{div}} \,{\bf z}\|_{L_2(\Omega)}^2\,.
    \end{equation}
   \par It was shown in \cite{AnufrievKorneevKostylev:2006, Korneev:2011} that the correction of arbitrary vector-function ${\bf z} \in {\bf H}(\Omega, {\mr{div}})$ into the vector-function $\bi \tau$, satisfying the balance/equilibrium equations, can be done by quite a few
   rather simple techniques. In particular, it is true for the correction of the flux vector-function $\nabla u_{\mr{fem}}$ into $\bi \tau(u_{\mr{fem}})$.  This allows to implement the a posteriori bound  $|\!\!\,\!|\!\!\,\!| v-u|\!\,\!\!|\!\,\!\!|\le \| {\bf A}\nabla\,u_{\mr{fem}} +\bi \tau(u_{\mr{fem}}) \|_{{\bf A}^{-1}}^2$ or the bound with the additional free vector-function
     in the right part, which we present below. For simplicity, we restrict considerations to the same homogeneous Dirichlet problem for the Poisson equation in a two-dimensional convex domain.  Let $T_k$ be the projection of the domain $\Omega$ on the axis $x_{3-k}$ and the equations of the left and lower parts of the boundary be $x_k=a_k(x_{3-k}),\,\, x_{3-k}\in T_k$.  If $\beta_k$
    are arbitrary bounded functions  and $\beta_1+\beta_2\equiv 1$,  then according to \cite{AnufrievKorneevKostylev:2006, Korneev:2011}
   \begin{equation}\label{yappi-1}
   \begin{array}{l}
    \|\nabla (v-  u) \|_{ {\bf L}_2(\Omega)} \le
      \|\nabla v + {\bf
    z} \|_{{\bf L}_2(\Omega)} +\\
    \vspace{-3mm}
    \\
      \sum_{k=1,2}\|\int_{a_k(x_{3-k})}^{x_k}
    \beta_k(f-\nabla \cdot{\bf z})
    (\eta_k,x_{3-k})\,d\eta_k \|_{ L_2(\Omega)}\,.
    \end{array}
   \end{equation}
\par In (\ref{yappi-1}) on the right  we have integrals from the residual and this helps to make the majorant more accurate.
  Besides there is an additional free function  $\beta_1 $  or $\beta_2 $  and it's  right choice (for instance, with the use of the found approximate solution $v$) can accelerate the process of the minimization of the right part.  Nevertheless, the majorant
(\ref{yappi-1}), as well as majorants (\ref{rep-frol}),  (\ref{Chur-1}), are not consistent.  Since it is practically obvious, below we discuss this matter   very briefly.

   \par  The inconsistency is the most clearly visible  for finite element methods of a higher smoothness. Let us turn to
    (\ref{rep-frol}) in the case when $v=u_{\mr{fem}} \in C^1(\omega) \cap H^2(\Omega)$ and $f\in H^1(\Omega)$. Therefore,
    $u\in H^3(\Omega)$, and the unimprovable a priori convergence estimates $\|u-v\|_{H^k(\Omega)}\le c h^{3-k}\|u\|_{H^3(\Omega)}$, $k=0,1,2$, hold with the mesh parameter $h$ under assumption that the finite element assemblage satisfies the conditions of the generalized quasiuniformity  \cite{Ciarlet:1978, Korneev:77}. For the latter conditions see, {\em{e. g.}},  Section 3.2 in
      \cite{KorneevLanger:2015}.     Now we see that the left part of (\ref{rep-frol}) is estimated from above with the order $h^4$.  One can set ${\bf z}=-\nabla v$  making  the first term in the right part equal to zero. At the same time,    the second term in the right part is estimated from above only with the order  $h^2$. More over since the estimates of the convergence are exact there are functions $f\in H^1(\Omega)$ for which the second term is estimated with the order  $h^2$ from below. The proofs of the inconsistency of the majorants (\ref{Chur-1}),
    and  (\ref{yappi-1}) are also straightforward.

     \par  If the FEM belongs to the class $C$, then we can use the so called recovered flux  ${\bf z}=\wt{\bf z}(u_{\mr{fem}})$, whose components  $\wt{z}_k(u_{\mr{fem}})$ in the simplest case are defined as functions of the same finite element space, to which belongs $u_{\mr{fem}}$.  Several cheap averaging procedures were developed for defining the  nodal parameters of
     $\wt{z}_k$, which provide at least the same orders of accuracy  for $(\bo u/\bo x_k -\wt{z}_k)$  and
     $\bo (u - u_{\mr{fem}}) /\bo x_k $, see  \cite{LiZhang:1999, AinsworthOden:00, BabuskaStrouboulis:01, BabuskaWhitemanStrouboulis:11}.  If $f\in L_2(\Omega)$, then   the order of the left part  is by the multiplier $h^2$  higher again than the order of the right part.

    \par Let $ \sigma_*$ be the value from the inequality
  \begin{equation}\label{Aub-1}
     \| u - v \|_{L_2(\Omega)} ^2 \le  \sigma_*^{-1}  \| u - v  \|_{{\bf A}} ^2\,.
   \end{equation}
 There is the multiplier $1/\sigma$ before the  second norm  in the right part of  (\ref{Aubin-1}). In view of this, it can be shown that at $\sigma \ge \sigma_*$ the Aubin's majorant is consistent for approximate  solutions by FEM, if $ \sigma_*^{-1}\le c_\dag h^2,\,\,c_\dag={\mr{const}}$, and some natural conditions are fulfilled, see Lemma~1 in the next section.  However, at
$\sigma\ll \sigma^*$ the consistency deteriorates and with $\sigma$ tending to zero the majorant  becomes meaningless.

 \vspace{0.2cm}
 \section{\normalsize Consistent  error majorant for any nonnegative
    reaction coefficient
 }
 \setcounter{equation}{0}
 \label{Se:cons}  
 \vspace{0.2cm}

\par We start from the presentation of a guaranteed robust error majorant valid   for all  $\sigma\in [0,\infty)$, which at application to the  FEM solutions is consistent.
 \vspace*{-1.4mm}
 \begin{theorem}\label{Th: K3}
  Let $\Gamma_D=\bo \Omega$, the conditions of Theorem~\ref{Th:Aubin-1} be fulfilled,
   and $\sigma_*$  satisfy the inequality
    (\ref{Aub-1}).
   Then
 \begin{equation}\label{K-23}
               |\!\!\,\!|\!\!\,\!| v-u|\!\,\!\!|\!\,\!\!|^2\le {\mc M}(\sigma,f,v,{\bf z})= \Theta
                             \Big[ ]\!| {\bf A}\nabla\,v +{\bf z} |\![_{{\bf A}^{-1}}^2+
            \theta    \|f-\sigma v - {\mr{div}} \,{\bf z}\|_{L_2(\Omega)}^2 \Big]\,,
 \end{equation}
  where for $\kappa=\sigma/\sigma_*$
 \begin{equation}\label{K-23-1}
 \begin{array}{ll}
  \Theta=\left\{  \begin{array}{ll}
   2/(1+\kappa), \q &\forall\, \sigma\in [0,\sigma_*] \\
    \vspace{-5mm}
    \\
    1,    &\forall\, \sigma> \sigma_*
    \end{array} \right\}, &
     \theta=\left\{  \begin{array}{ll}
 1/\sigma_*,  \q   &\forall\, \sigma\in [0,\sigma_*] \\
    \vspace{-5mm}
    \\
 1/\sigma,  \q   &\forall\, \sigma> \sigma_*
 \end{array} \right\}\,.
 \end{array}
 \end{equation}
 %
  \end{theorem}
  \begin {proof} Obviously, for $\sigma \ge \sigma_*$ the  majorant (\ref{K-23}), (\ref{K-23-1})   coincides  with the majorant of Aubin. Consequently, it is necessary to consider only the case $\sigma < \sigma_*$.  For simplicity, in the proof we set  ${\bf A}={\bf I}$ and $\psi_D \equiv 0$. For the solution of the problem $u$,  arbitrary function $v\in \7{H}^1(\Omega)$ and
   vector-function ${\bf z} \in {\bf H}(\Omega, {\mr{div}})$,  we can write
    \begin{equation}\label{bound-1}
\begin{array}{c}
    |\!\!\,\!|\!\!\,\!| v-u|\!\,\!\!|\!\,\!\!|^2= \int_\Omega \big[{\bi \nabla} (v-u)\cdot \nabla (v-u) +
    \sigma (v-u)(v-u) \big]= \\
    \\ \int_\Omega \big[(\nabla v+{\bf z})\cdot \nabla (v-u) -  ({\bf z}+\nabla u)\cdot
    \nabla (v-u)+  \\
    \\  \sigma (v-u)(v-u) \big]\,.
  \end{array}
\end{equation}
  Integrating by parts the second summand in the right part  and implementing the inequality
   \begin{equation}\label{simple}
   a_1b_1+a_2b_2\le (a_1^2+\frac{1}{\sigma_*}a_2^2)^{1/2}(b_1^2+{\sigma_*}b_2^2)^{1/2}\,,
  \end{equation}
   we find out that
 \begin{equation}\label{K-25}
 \begin{array}{c}
    |\!\!\,\!|\!\!\,\!| v-u|\!\,\!\!|\!\,\!\!|^2  =\| \nabla (u-v)\|_{{\bf L}^2(\Omega)}^2+\sigma\|u-v\|_{ L^2(\Omega)}^2\le \\
    \vspace*{-3mm}
   \\\Big[ \| \nabla (v-w)\|_{{\bf L}^2(\Omega)}^2     +\frac{1}{\sigma_*}\|f-\sigma v + \Delta w\|_{L^2(\Omega)}^2\Big]^{1/2}\times\\
    \vspace*{-3mm}
   \\
    \Big[ \| \nabla (u-v)\|_{{\bf L}^2(\Omega)}^2+\sigma_*\|u-v\|_{ L^2(\Omega)}^2 \Big]^{1/2}\,.
  \end{array}
   \end{equation}
   The use of  $\beta \in (0,1]$ and  (\ref{Aub-1}) allows us to get
  \begin{equation}\label{K-26}
  \begin{array}{c}
   \| \nabla (u-v)\|_{{\bf L}^2(\Omega)}^2+\sigma_*\|u-v\|_{ L^2(\Omega)}^2=|\!\!\,\!|\!\!\,\!| u-v|\!\,\!\!|\!\,\!\!|^2+
   (\sigma_*-\sigma)\|u-v\|_{ L^2(\Omega)}^2\le \\
    \vspace*{-3mm}
   \\
 |\!\!\,\!|\!\!\,\!| u-v|\!\,\!\!|\!\,\!\!|^2+
   (\sigma_*-\sigma)\big[ \frac{\beta}{ \sigma_*}\| \nabla (u-v)\|_{{\bf L}^2(\Omega)}^2 + (1-\beta)\|u-v\|_{ L^2(\Omega)}^2 \big] =\\
    \vspace*{-3mm}
   \\
 \big[ 1+(\sigma_*-\sigma)\frac{\beta}{\sigma_*}  \big]\| \nabla (u-v)\|_{{\bf L}^2(\Omega)}^2+\big[ (1-\beta)(\sigma_*-\sigma)+\sigma \big]\|u-v\|_{ L^2(\Omega)}^2\,,
   \end{array}
 \end{equation}
   The value
    $
    \beta= {2}/(1+\kappa)
    $
    makes the relation of the multipliers before the second and first norms on the right of (\ref{K-26}) equal to $\sigma$.
     Substituting it into (\ref{K-26}) and then (\ref{K-26}) into (\ref{K-25}) yields
   \begin{equation}\label{K-27}
       |\!\!\,\!|\!\!\,\!| v-u|\!\,\!\!|\!\,\!\!|^2 \le \frac {2}{1+\kappa} \Big[ \| \nabla (v-w)\|_{{\bf L}^2(\Omega)}^2     +\frac{1}{\sigma_*}\|f-\sigma v + \Delta w\|_{L^2(\Omega)}^2\Big]^{1/2} |\!\!\,\!|\!\!\,\!| v-u|\!\,\!\!|\!\,\!\!| \, ,
         \end{equation}
   which is equivalent to (\ref{K-23}) in the case of ${\bf A}={\bf I}$.  
  \end{proof}

  \par As was noted above, for $\sigma \ge \sigma_*$ the  majorant (\ref{K-23}), (\ref{K-23-1})   coincides  with the majorant of Aubin. In the contrast to Aubin's majorant, for all $\sigma \ge 0$ it is well defined and, more over,
      belongs to the class of  consistent majorants when applied to the  solutions by the  finite element method  satisfying quite natural conditions. Before formulating  the respective result in Lemma~1 below, we  briefly discuss these conditions.

    \par It is assumed that the finite element space  ${\mbb V}_h(\Omega)$,
     ${\mbb V}_h(\Omega) \!\!\subset\!\! C(\Omega)\!\cap\!  H^1(\Omega)$, is induced by the assemblage of the finite elements, in general curvilinear, which satisfy the generalized conditions of quasiuniformity  with the mesh parameter $h$, see {\em{e. g.}}
      \cite{Korneev:77,  KorneevLanger:2015}, and  $\7{\mbb
     V}_h(\Omega)=\{v\in {\mbb V}_h(\Omega):\,\,v|_{\bo \Omega}=0 \}$.  For simplicity, we  consider  the FEM of
       the first order of  accuracy, {\em i. e.}  with finite elements associated with the triangular linear and square bilinear reference elemens.
        If $f\in L_2(\Omega)$, boundary $\bo \Omega$ and the coefficients of the matrix ${\bf A}$ are sufficiently smooth,
         then the following convergence estimates  can be proved:
        \begin{equation}\label{fem-conv}
        \begin{array}{l}
        \| u - u_{\mr{fem}} \|_{k,\Omega}\le c_{k,l} h^{l-k}\|u
        \|_{l,\Omega}\,, \q
                 k=0,1\,, \q l=1,2\,,  \\
       \vspace*{-4mm}
                 \\
          \| u - u_{\mr{fem}} \|_{0,\Omega} \le {\mr{min}}[ c_\circ c_{0,2} h^2,\sigma^{-1}] \|f
            \|_{0,\Omega}\,,\q \forall\,\sigma\ge 0\,,
       \end{array}
          \end{equation}
          where $\|\cdot\|_{k,\Omega}$ are the norms in the space  $L_2(\Omega) $ for $k=0 $ and in the spaces $H^k(\Omega) $  for $k>0$,  whereas
           $c_\circ, c_{k,l}={\mr{const}}$.  If $\sigma=0$,  they are the well known FEM convergence estimates  for regular elliptic problems, see, {\em e. g.},
            \cite{OganesianRuhovets:1979, Ciarlet:1978, Korneev:77}.  If $\sigma\le c_\dag^{-1}h^{-2}$, the results of  \cite{BrambleXu:91}  on elliptic projections  in the space $L_2$  together with the fact that at least $u\in H^2(\Omega)$ can be used for their proof.  Indeed, it is easily shown
           that
   \begin{equation}\label{2-f}
   \| u \|_{H^2(\Omega)} \le c_\circ \|f \|_{L_2(\Omega)},\,\,c_\circ={\mr{const}}\,,
   \end{equation}
   at any $\sigma \ge 0$, if it is true (with different constant) for  $\sigma=0$. The second bound  (\ref{fem-conv}) takes additionally into account the bound $\| u - u_{\mr{fem}} \|_{0,\Omega} \le \sigma^{-1} \|f
            \|_{0,\Omega}$.

    \par For the use of a posteriori majorant (\ref{K-23}), (\ref{K-23-1}), one has to bound $c_\dag$. First we turn to the case  $\sigma=0$.  By means of Nitsche trick, see
   {\em{e.  g.}}  \cite{Ciarlet:1978, OganesianRuhovets:1979},  for $e_{\mr{fem}}=u - u_{\mr{fem}}$ it is proved the inequality
   \begin{equation}\label{Nitsche}
      \|\,  e_{\mr{fem}}\,  \|_{0,\Omega}  ^2 \le
       \|\, e_{\mr{fem}}\, \|_{\bf A} \, \| \, \phi-\phi_{\mr{int}}\,\|_{\bf A}      \, ,
         \end{equation}
     where $\phi$ is the solution of the boundary value problem  ${\mc L}\phi=e_{\mr{fem}},\,\, \phi |_{\bo \Omega} =0, $
     and, according to (\ref{2-f}),   $\phi \in H^2(\Omega)$, whereas      $\phi_{\mr{int}}$ is the interpolation of $\phi$ from the finite element  space $\7{\mbb  V}_h(\Omega)$.  Combining the approximation error bounds
        \begin{equation}\label{fem-approx}
        \| \phi - \phi_{\mr{int}} \|_{k,\Omega}\le \hat{c}_{k,l} h^{l-k}\|\phi
        \|_{l,\Omega}\,, \q
                 k=0,1\,, \q l=1,2\,,
          \end{equation}
            (\ref{Nitsche}) and (\ref{2-f}) yields (\ref{Aub-1}) with $ \sigma_*^{-1}\le c_\dag h^2$  and
            \begin{equation}\label{Aub-kkor}
            c_\dag=\mu_2\hat{c}_{1,2}^2c_\circ^2\,.
            \end{equation}
            %
            %
    \par Now we will use the notations $e_\sigma=e_{\mr{fem}}$  and $e_0$ for the errors of the finite element solutions of the 
  equations  ${\mc L}u + \sigma u=f$  and ${\mc L}u = f_1$, respectively,  with the first  boundary condition $u|_{\bo \Omega} =0$ and $f_1=f-\sigma u$.  Since from the proof given above and the introduced definitions it follows  that
  \begin{equation}\label{Aub-k}
     \|  e_0\|_{0,\Omega} ^2 \le c_\dag h^2 \| e_0  \|_{{\bf A}} ^2\,, \Q  \|  e_\sigma\|_{0,\Omega}\le\|  e_0\|_{0,\Omega}\,, \q
     \| e_0  \|_{{\bf A}} \le \| e_\sigma  \|_{{\bf A}} \,,
   \end{equation}
  we come to (\ref{Aub-1}) of the form
  \begin{equation}\label{Aub-kk}
     \|  e_{\mr{fem}}\|_{0,\Omega} ^2 \le c_\dag h^2 \| e_{\mr{fem}}  \|_{{\bf A}} ^2
  \end{equation}
  with the same, as in (\ref{Aub-kkor})  and (\ref{Aub-k}), constant  $c_\dag$.  Accordingly, at  $\sigma \in [0, 1/(c_\dag h^2)]$
  the bound (\ref{K-23})  for the finite element solutions   can be rewritten as
 \begin{equation}\label{K-23k}
 \begin{array}{l}
               |\!\!\,\!|\!\!\,\!| u_{\mr{fem}}-u|\!\,\!\!|\!\,\!\!|^2\le {\mc M}_{\mr{fem}}(\sigma,f,u_{\mr{fem}},{\bf z})=\\
               \vspace*{-1mm}
               \\
                \frac{2}{1+c_\dag h^2\sigma}
                             \Big[ ]\!| {\bf A}\nabla\,v +{\bf z} |\![_{{\bf A}^{-1}}^2+
            c_\dag h^2   \|f-\sigma v - {\mr{div}} \,{\bf z}\|_{L_2(\Omega)}^2 \Big]\,,
 \end{array}
 \end{equation}

       \par   The construction of the recovered vector-function ${\bf z}=\wt{\bf z}(u_{\mr{fem}}) \in {\bf H}(\Omega,  {\mr{div}})$ can be performed with the use of  the  finite element fluxes
          $-{\bf A} \nabla u_{\mr {fem}}$. The convergence bounds (\ref{fem-conv}) lead  to the conclusion that at any $\sigma\ge 0$  the same recovery techniques can be used,  which are used
            for regular elliptic problems \cite{AinsworthOden:00, BabuskaStrouboulis:01, BabuskaWhitemanStrouboulis:11, LiZhang:1999}. They allow  to obtain such vector-functions ${\bf z}$
         with componetnts  satisfying the inequalities
        \begin{equation}\label{fem-aver-conv}
           \begin{array}{l}
           \| \wt{\bf  z}(\phi) \|_{ {\bf L}_2(\Omega)}\le
          \hat{c}\|\nabla\phi\|_{{\bf L_2}(\Omega)}\,,\Q \forall \phi\in  {\mbb V}_h(\Omega)\,,\\
          \vspace*{-4mm}
          \\
          \|\nabla u + \wt{\bf  z}(u_{\mr{fem}}) \|_{{\bf L}_2(\Omega)}\le
          \hat{c}_{1}h^{l-1}\|u\|_{H^l(\Omega)}\,, \q l=1,2\,,\\
          \vspace*{-4mm}
          \\
           \|\Delta u + \nabla \cdot \wt{\bf  z}(u_{\mr{fem}}) \|_{ L_2(\Omega)}\le
          \hat{c}_{2}\|u\|_{H^2(\Omega)}\,.
         \end{array}
         \end{equation}
      \begin{lemma}  \label{lemma}
     Let $\Gamma_D=\bo \Omega$, $\psi_D \equiv 0$ and $f\in L_2(\Omega)$, the finite element assemblage
     satisfy the conditions of the generalized quasiuniformity, and
     the convergence estimates (\ref{fem-conv}) hold. Let also  $v=u_{\mr{fem}}$
        and
         for the vector-function  ${\bf z}=\wt{\bf z}(v)$, obtained by the application of the recovery technique
          to the  finite element fluxes
          $-{\bf A}\cdot \nabla v$,
          the inequalities (\ref{fem-aver-conv}) hold.  Then for $ \sigma_*^{-1}\le c_\dag h^2$ with $ c_\dag $ from (\ref{Aub-kkor})
           and any $\sigma\ge 0 $
          we have
      \begin{equation}  \label{consist}
             {\mc M}(\sigma,f,u_{\mr{fem}},\wt{\bf z})  \le Ch^{2}\|f\|_{L_2(\Omega)}
               \end{equation}
            with the constant     $C$  independent of  $\sigma$ and $h$.
         \end{lemma}
      \par  In fact, the recovered flax is defined in such a way  that at least to have the same orders of accuracy with the flax
  defined by the finite element solution or the same orders of accuracy in the unimprovable  a priori error bounds. More over
  the superconvergence recovery technique (SPR)  demonstrated ability to provide the superconvergent recovery on regular meshes
  and recovery with  much improved accuracy on general meshes. The mathematical analysis approving   this phenomena for some finite element methods can be found in \cite{Zhang:96, LiZhang:1999}. At the same time, alongside with (\ref{consist}) it is not difficult to establish the consistency of the majorant ${\mc M}(\sigma,f,v,{\bf z})$ with the a priori error bounds for finite element methods of higher order of accuracy.

  \vspace{0.2cm}
   \section{\normalsize  Concluding remarks} \label{Se: Conclud} 
   \setcounter{equation}{0}
    \vspace*{0.2cm}
    \par Theorem 2 and Lemma 1  are formulated for the first boundary value problem. If the natural boundary condition is posed on the part of the boundary, then the necessary changes of these results are illustrated by Theorem 1. Namely, vector-functions ${\bf z}$, $\wt{\bf z}(v)$  and functions $w$, $\wt{w}(v)$  should satisfy  this boundary condition.  However, in general the finite element spaces, to which these vector-functions and functions belong, do not allow  to satisfy  the boundary condition exactly. Therefore, they must be approximated in the corresponding trace spaces, and as a consequence the additional terms estimating  influence of the approximation  appear in the majorants.  The technique of the estimating   such additional terms is common for a posteriory bounds of different types and can be found, {\em e. g.}, in   \cite{AnufrievKorneevKostylev:2006}, see   Remark 4.5.

\par Results of the paper can be expanded upon more general elliptic equationns of orders $2n,\,n\!\!\ge\! 1$  and, in particular, to those described un Theorems 1.2, 1.4, 1.6 of ch. 10 in  \cite{Aubin:72}. One of them is the equation
  $ {\mc L}_n u+\sigma u=f$ with the differential operator
  $$
  {\mc L}_n u=   \sum_{|{\bf q}|,|{\bf p}|=n}(-1)^{|{\bf q}|}D^{\bf q}a_{{\bf q},{\bf p}}(x)D^{\bf p}u\,,
  $$
   where $D^{\bf q}v=\bo^{|{\bf q}|}v/\bo
x_1^{q_1}\bo x_2^{q_2}\dots \bo x_m^{q_m},\,\,
{\bf q}=(q_1,q_2,\dots,q_m)$,\,  $q_k$ are nonngative whole numbers, $|{\bf q}|=q_1+q_2+\dots +q_m $,  ${\bf A}=\{a_{{\bf p},{\bf q}}\}_{|{\bf p}|,|{\bf q}|=
  n}$ is the matrix with the sufficiently smooth coefficients, satisfying the inequalities
      $\UL{\mu }{\bf I}\le{\bf A}\le \OL{\mu} {\bf I},\,\, 0\le \UL{\mu }, \OL{\mu}={\mr{const}},\,\, \forall\, x\in \OL{\Omega}$. Here the inequality ${\bf B}\le {\bf C}$  for two nonnegative matrices ${\bf B}$ and ${\bf C}$ of the same dimension assumes that $({\bf C}-{\bf B}) $ is a nonnegative matrix.
    \par For definiteness, we turn to the case of the first boundary condition  $\bo ^k u/\bo \nu^k=0,\,\, k=0,1,\dots,(n-1),\,\, \forall\,\,x\in \bo \Omega$, where $\nu$ is the distance to the boundary along the  normal ${\bi \nu}$, and define
    the norm
    $$
         \|v\|_{\bf A}=\Big(\sum _{|{\bf q}|,|{\bf p}|=n} \int_{\Omega} a_{{\bf p},{\bf q}} (D^{\bf q}v)D^{\bf p}v dx\Big)^{1/2}\,.
    $$
       Under  the well known conditions, the value $\sigma_*= \|u-v\|_{\bf A}^2/\|u-v\|_{L_2(\Omega)}^2$ for the FEM solutions $v=u_{\mr{fem}}$
    is estimated from below as  $\sigma_*\ge 1/ ( c_{n,\dag}h^{2n}),\,\,c_{n,\dag}={\mr{const}}$.
        The  bound of the  identical to (\ref{K-23}) form retains, if for the introduced differential operator ${\mc L}_n$, matrix ${\bf A}$ and number $\sigma_*$ the norms $|\!\!\,\!|\!\!\,\!| \cdot|\!\,\!\!|\!\,\!\!|,\,\,]\!| \cdot |\![_{{\bf A}^{-1}}$ and the functions   $\Theta(\kappa),\, \theta(\kappa)$ are correspondingly defined and the vector-function $\nabla v $ and the function  ${\mr{div}} \,{\bf z}$  are replaced by ${\mc D}v=\{D^{\bf p}v \}_{|{\bf p}|=n}$ and $\,{\mc D}^*{\bf z}=\sum_{|{\bf q}|=n} (-1)^{|{\bf q}|} D^{\bf q}{ z}^{({\bf q})}$, respectively,
     where ${ z}^{({\bf q})}$  are components of the vector ${\bf z}$.

\vspace{0.2cm}
   \section{\normalsize Acknowledgemens}\label{Se:Ack}
      \vspace*{0.2cm}
    \setcounter{equation}{0}
\par  The author expresses his    sincere gratitude to professor M.R. Timerbaev and  professor M.M. Karchevsky for helpful discussions and
    improvements of the text and to doctor V. S. Kostylev  for  making-ready the text files.  Research was
supported  by the grants from the Russian Fund of Basic Research. The  author has  been partially
           supported by the Johann  Radon Institute for Computational and Applied Mathematics (RICAM)
           of the  Austrian  Academy of Sciences during his research visits at Linz.


\end{document}